\documentclass[conference]{IEEEtran}
\usepackage{fancyhdr}
\usepackage{cite}
\usepackage[tight,footnotesize]{subfigure}
\usepackage{url}

\usepackage{amsmath,amsthm}
\usepackage{amssymb}
\usepackage{amsmath}
\usepackage{color}
\usepackage{graphics}
\usepackage{graphicx}
\usepackage{subfigure}
\usepackage{algorithm}
\usepackage{multirow}

\def \bbb{\color{blue}}

\def \kkk{\color{black}}
\def \R{\mathbb{R}}

\def \E{\mathcal{E}}
\def \I{\mathcal{I}}

\def \S{\mathcal{S}}

\def \D{\Delta}

\def \I{\mathcal{I}}

\def \g{\gamma}

\def \a{\alpha}

\def \s{\sigma}
\def \St{\S_{\tau}}

\newtheorem{theorem}{Theorem}
\newtheorem{remark}{Remark}
\newtheorem{lemma}{Lemma}

\begin{document}
%
\title{Stability of Switched Linear Systems under Dwell
Time Switching with Piece Wise Quadratic Functions}

\author{\IEEEauthorblockN{Masood Dehghan}
\IEEEauthorblockA{Advanced Robotics Centre\\
National University of Singapore\\
1 Engineering Drive 3, Singapore 117580, Singapore\\
Email: masood@nus.edu.sg}
\and
\IEEEauthorblockN{Marcelo H. Ang}
\IEEEauthorblockA{Department of Mechanical Engineering\\
National University of Singapore\\
9 Engineering Dr 1, Singapore 117576, Singapore\\
Email: mpeangh@nus.edu.sg}
}


%


\maketitle

\thispagestyle{fancy}
\fancyhead{}
\lhead{}
\lfoot{}
\cfoot{}
\rfoot{}
\renewcommand{\headrulewidth}{0pt}
\renewcommand{\footrulewidth}{0pt}

\begin{abstract}
This paper provides sufficient conditions for stability of switched linear systems under dwell-time switching.  Piece-wise quadratic  functions are utilized to characterize the Lyapunov functions and
bilinear matrix inequalities conditions are derived for stability of switched systems. By increasing the number of  quadratic functions, a sequence of upper bounds of the minimum dwell time is obtained.
Numerical examples suggest that if the number of  quadratic functions is sufficiently large, the sequence may converge to the  minimum dwell-time. \kkk
\end{abstract}


%
\IEEEpeerreviewmaketitle

\section{Introduction}
\label{sec:into}

This paper  investigates to the stability  of  switched linear systems:
\begin{align}
        & \dot x(t) = A_{\sigma(t)} \, x(t), \label{eqn:1a}
\end{align}
where $x(t) \in \R^n$ is the state variable and $\sigma(t):\R^+
\rightarrow \I_N:=\{1,\cdots, N\}$ is a time-dependent switching
signal that indicates the current active mode of the system among $N$
possible modes in $\mathcal{A}:=\{A_1, \cdots, A_N\}$.
All matrices $A_i, i = 1, \cdots, N$ are assumed to be Hurwitz.

This class of systems has been widely investigated in the last decade, due to their importance both in the theoretical context and in engineering applications, see e.g. the  recent surveys  \cite{Branicky1998,book_Liberzon,Shorten2007,HaiLin2009}.

The stability problem is  one of the most important issue associated with the study of such systems. This problem has been addressed mainly using the theory of Lyapunov functions (LFs), to find  conditions under which  the system preserves stability.
For example, the origin of system (\ref{eqn:1a}) is
stable under arbitrary switching upon the existence of a common
quadratic Lyapunov function
\cite{book_Liberzon},  switched
Lyapunov functions \cite{switched_Lyapunov2002}, multiple
Lyapunov functions \cite{Branicky1998,dehghan2014computation},
composite quadratic functions \cite{Hu2008}
 or polyhedral Lyapunov functions
\cite{Blanchini20072,Dehghan-sat}.
It is known  that existence of a piece-wise linear (polyhedral) LF \cite{Blanchini20072}
or a piece-wise quadratic  LF \cite{Hua2010}  is both necessary and sufficient  for asymptotic stability of
system (\ref{eqn:1a}) under arbitrary switching.
This implies that the class of polyhedral functions or piece-wise quadratic functions are universal for characterization of stability of switched linear systems under arbitrary switching.

Another condition for stability
is that based on dwell-time consideration. When all $A_i$'s are stable,
stability of the origin can be ensured if the time duration spent
in each subsystem is sufficiently long \cite{book_Liberzon}.
Upper bounds of the minimal dwell-time needed have also appeared
\cite{Geromel2006,Blanchini-2010,our11,Chesi2012,our22}.
Finding the minimum dwell time is known to be a hard problem and the research has been shifted to find testable conditions for the computation of upper bounds to the minimum dwell time.
In \cite{Wirth-2005} it is shown that stability under dwell time implies the existence of multiple Lyapunov norms, but the result is not constructive.
In \cite{Blanchini-2010} a necessary and sufficient condition for  stability of (\ref{eqn:1a})  in terms of piecewise linear (polyhedral) LF is provided, however construction of such polyhedral LF are not easy.
Alternatively,  \cite{our11,our33} introduce the concept of dwell-time(DT)-contractive sets and show that existence of a polyhedral DT-contractive set is both necessary and sufficient for stability of (\ref{eqn:1a}). An algorithm for computation of DT-contractive sets is also proposed for discrete-time systems, however computation of such sets for continuous systems is still lacking.
In \cite{Chesi2012}, polynomial  functions are used to characterize the LF and the problem is  formulated  as a set of LMIs to compute upper bounds of minimum dwell time.

In this paper, we use piece-wise quadratic functions to characterize the Lyapunov functions and provide stability conditions in terms of bilinear matrix inequalities (BMIs).
In the limiting case where the dwell-time approaches zero, system is under arbitrary switching, and the proposed conditions retrieve the results
 appeared in the
literature  \cite{Hua2010}.
Hence, this work can also be seen as a generalization of those obtained for
arbitrary switching systems.
It turns out by increasing the number  of the  quadratic functions that characterize the Lyapunov function, the proposed  conditions has more degree of freedom and can be used to determine the minimal dwell-time needed for
stability of (\ref{eqn:1a}).

The rest of this paper is organized as follows. This section ends
with a description of the notations used. Section \ref{sec:prelim}
reviews some standard terminology and results for switching
systems.  Section \ref{sec:main} shows the main results on the
characterization of the Lyapunov functions with piece-wise quadratic functions  for system
(\ref{eqn:1a}). An algorithmic procedure  for  computation of sequence of upper bounds of the minimum dwell time needed for stability is also presented in this section.
Sections \ref{sec:example} and  \ref{sec:conclusions}
contain, respectively, numerical examples and conclusions.

The following standard notations are used. $\R^+$ is the set of
non-negative real numbers.
Positive definite (semi-definite) matrix, $P \in \R^{n
\times n}$, is indicated by $P \succ 0 (\succeq 0)$ and $I_{n}$ is the $n \times n$ identity matrix.
Given a $P \succ 0$,
 $\E(P):=\{ x: x^TPx \le 1 \}$.
Other notations are introduced when needed.

\section{Preliminaries}
\label{sec:prelim}

This section reviews definitions of dwell-time, admissible switching sequences, piece-wise quadratic functions and preliminary results on stability of (\ref{eqn:1a}) under dwell-time switching.  These definitions
have appeared in  prior papers (see, e.g. \cite{our11,Hu-2003,Hua2010,Chesi2012}) but are repeated here for completeness
and for setting up the needed notations and results.

Denoting by $t_k$, $k = 0,1,2, \cdots $ the switching instants, we assume that the following dwell-time restriction is imposed on the switching sequence
$\s$, i.e. $\s \in \S_\tau$,
\begin{align}
\S_\tau = \{ \s(t) :  t_{k+1} - t_k \ge \tau \}
\end{align}
where $\S_\tau$  is the  set of admissible switching signals that satisfies the dwell-time  $\tau \ge 0$ restriction.
Note that $\s(t)$ is a piecewise constant function, in the sense that $\s(t) =\s(t_k) $ for $ t \in  [t_k,t_{k+1})$.
The minimum dwell-time, $\tau_{min}$, is defined as the minimum $\tau$ ensuring asymptotic stability of system (\ref{eqn:1a}) for all possible $\s \in \St$.
More specifically, it is defined as
\begin{align*}
\tau_{min} := \inf \{ \tau \ge 0 : (\ref{eqn:1a}) \textrm{ is Asymptotically stable,} \forall \s(t) \in \St \}
\end{align*}

\subsection{Piece-wise Quadratic Functions}
For a positive semidefinite function $V: \R^n \rightarrow [0, \infty)$, denote its $1$-level set as
\begin{align*}
L_V := \{ x \in \R^n: V(x) \le 1 \}
\end{align*}

The one sided directional derivative of $V(x)$ is defined with respect to two variables $x$ and $\zeta$, where $\zeta$ specifies the direction of increment or motion
\begin{align*}
\dot V (x ; \zeta) := \lim_{h \downarrow 0} \frac{ V(x+\zeta h) - V(x) }{h}
\end{align*}
where  $h \downarrow 0$ denotes decreasing to $0$.

Given $m$ positive definite matrices $P_r \succ 0$, $r = 1, \cdots, m$, a piece-wise quadratic function can be obtained by \cite{Hu-2003}:
\begin{align}
\label{eqn:max}
 V_{max} (x) :=  \max \{  x^T P_r \, x : r = 1 , \cdots, m \}
\end{align}
as the pointwise maximum of $m$ functions $x^T P_r x$, $r=1, \cdots, m$ and its 1-level set is the intersection of the ellipsoids $\E(P_r)$, i.e. $L_{V_{max}} = \cap_{r=1}^{m} \E(P_r)$.

The directional derivative of $V_{max}$ along the $i$-th mode of system (\ref{eqn:1a}) is \cite{Hu-2003}:
\begin{align}
\label{eqn:dot-max}
\dot V_{max} (x\, ; A_i x) := \max \big\{ & x^T ( A_i^T P_s + P_s A_i )x :  \nonumber \\
& s \in \{ s: V_{max}(x)=x^T P_s \, x \} \big\}
\end{align}

\subsection{Stability Results}
We start from the following theorem which states the necessary and sufficient condition for stability of (\ref{eqn:1a}) with dwell-time $\tau$.

\begin{theorem}
\label{thm:1}
\cite{Wirth-2005} System (\ref{eqn:1a}) is asymptotically stable in $\St$,  if and only if there exist continuous functions $V_i(x)$, associated to each mode $i$, such that
\begin{subequations}
\label{eqn:NS}
\begin{align}
& V_i(x) > 0  & \forall x \ne 0, \forall i \\
& \dot V_i (x\, ; A_i x) < 0  & \forall x \ne 0, \forall i  \\
& V_j(e^{A_i \tau}x) < V_i(x) & \forall x \ne 0, \forall i \neq j
\end{align}
\end{subequations}
\end{theorem}

The above theorem shows  that stability of the system (\ref{eqn:1a}) the equivalent the existence of a mode-dependent Lyapunov function,  $V_{\s(t)}(x)$, which is strictly decreasing for non-switching times, i.e. $t \neq  t_k$, and it is strictly decreasing at the switching instances, i.e.,
$V_{\s(t_{k+1})}  \big( x(t_{k+1})  \big) < V_{\s(t_{k})} \big( x(t_{k}) \big)$,
however the the above theorem is not constructive.
In \cite{Geromel2006}, quadratic functions are  used to characterize the $V_i$'s,
but the conditions are only sufficient for stability.
In \cite{Blanchini-2010}, polyhedral (piece-wise linear) LFs are
used to characterize the $V_i$'s and it is shown that piece-wise linear functions that satisfy condition (\ref{eqn:NS}) are both necessary and sufficient for stability of (\ref{eqn:1a}). However, the conditions obtained are nonlinear and cannot be solved efficiently.
Motivated by the fact that piece-wise quadratic functions are universal for characterization of stability similar to piece-wise linear functions \cite{Hu-2003},
the following section derives the stability conditions using   piece-wise quadratic functions.

\section{Main Results}
\label{sec:main}

Given a positive integer $m$, a piece-wise quadratic function characterized  by $m$ quadric functions is considered as the candidate Lyapunov function $V_i$ in Theorem \ref{thm:1}, namely
$V_i (x) = \max \{ x^T P_{i,r} \, x : r = 1, \cdots, m \}$, $i \in \I_N$.
Using S-procedure, conditions of Theorem \ref{thm:1} can be converted into matrix inequalities.
The following theorem provides a sufficient condition for stability of (\ref{eqn:1a}) under dwell-time switching.

\begin{theorem}
\label{thm:2}
Assume that, for a given $\tau > 0$ and positive integer $m$,  there exist scalars $\alpha_{irs} > 0$, 
$\gamma_{jqirs} \ge 0$, $\sum_{s=1}^{m} \gamma_{jqirs} < 1$ such that
\begin{align}
& P_{i,r} \succ 0  \qquad \qquad \qquad \qquad \qquad  \forall i , r = 1,\cdots, m  \label{eqn:7}\\
& A_{i}^{T} P_{i,r} + P_{i,r} A_{i}   \prec  \sum_{\substack{s=1\\s\ne r}}^{m} \alpha_{irs} (P_{i,s} - P_{i,r})  \nonumber \\
& \qquad \qquad \qquad \qquad \qquad \qquad \qquad    \forall i , r = 1,\cdots, m \label{eqn:8}\\
& e^{A_{i}^{T} \tau }  P_{j, \bbb q \kkk} \, e^{A_{i} \tau }  \prec P_{i,r} + \sum_{\substack{s=1\\s\ne r}}^{m} \gamma_{jqirs} (P_{i,s} - P_{i,r} )   \nonumber \\
& \qquad \qquad \qquad \qquad \qquad \qquad  \forall i \neq j,
 r = 1,\cdots, m  \nonumber \\
&  \qquad \qquad \qquad \qquad \qquad \qquad   \quad q = 1,\cdots, m   \label{eqn:9}
\end{align}
Then, system (\ref{eqn:1a}) is asymptotically stable for every  $\s \in \St$.
\end{theorem}

\begin{proof}
Let $V_i(x) = \max \{ x^T P_{i,r} \, x : r = 1, \cdots, m \}$, $i \in \I_N$. We have to show that conditions \ref{eqn:NS}(a)-(c) are satisfied.
a) Obviously, (\ref{eqn:7}) implies that $V_i (x) > 0$ for all $x \ne 0$ and for all $i \in \I_N$.
b) From (\ref{eqn:8}) and (\ref{eqn:dot-max}), it follows that  $\dot V_i (x ; A_i x) < 0$ for all $x \ne 0$ and for all $i \in \I_N$, see \cite{Hu-2003} for details.
c) Without loss of generality consider $V_j(x) = \max \{ x^T P_{j,q} \, x : q = 1, \cdots, m \}$, $j \neq i$. We have to show that
$V_j(e^{A_i \tau} x) < V_i(x)$.  To this end, consider $(e^{A_i \tau} x)^T P_{j,q} \, (e^{A_i \tau} x)$ for any $q = 1, \cdots, m$.
For every $x$ such that $x^T P_{i,r} \, x \ge  x^T P_{i,s} \, x$, $s = 1, \cdots , m$, $V_i(x) = x^T P_{i,r} x$. This and
(\ref{eqn:9}) together impliy that
\begin{align*}
x^T (e^{A_{i}^{T} \tau }  P_{j,  q} \, e^{A_{i} \tau } ) x  &< \underbrace{x^T P_{i,r} x}_{V_i(x)} +  \sum_{\substack{s=1\\s\ne r}}^{m} \gamma_{jqirs} \underbrace{x^T(P_{i,s} - P_{i,r} )x}_{\le 0} \\
&< V_i(x)
\end{align*}
Thus, $V_j(e^{A_i \tau} x) = \max_q \{ x^T e^{A_i^T \tau}  P_{j,q} \, e^{A_i \tau} x  \}  < V_i(x)$. The same argument holds for all the other regions where $x^T P_{i,r} \, x \le  x^T P_{i,s}\, x$. Hence conditions  (\ref{eqn:NS}a)-(\ref{eqn:NS}c) of Theorem (\ref{thm:1}) are all satisfied and the proof is complete.
\end{proof}


\begin{remark}
Note that, for $m = 1$, $V_i(x) = x^T P_{i1}\, x$ and conditions (\ref{eqn:7})-(\ref{eqn:9}) become:  $P_{i,1} \succ 0, \,  \forall i$,
\begin{align*}
& A_{i}^{T} P_{i,1} + P_{i,1} A_{i}   \prec  0 & \forall i   \\
& e^{A_{i}^{T} \tau }  P_{j, 1} \, e^{A_{i} \tau }  \prec P_{i,1}    & \forall i \neq j
\end{align*}
and it retrieves the conditions appeared in \cite{Geromel2006}.
\end{remark}

For a given $m$, define   $\tau_{[m]}$ as the smallest upper bound of $\tau_{min}$ guaranteed by Theorem \ref{thm:2}, i.e.
$$ \tau_{[m]} :=\inf \{ \tau \ge 0: (\ref{eqn:7})-(\ref{eqn:9}) \textrm{ hold} \}.$$
The following result provides a key property of the conditions of Theorem \ref{thm:2}, which allows us to calculate $\tau_{[m]} $ via a bisection search where at each iteration the conditions (\ref{eqn:7})-(\ref{eqn:9}) are tested.

\begin{theorem}
\label{thm:3}
Assume that (\ref{eqn:7})-(\ref{eqn:9}) hold for some $\tau \ge 0$ and positive integer $m$. Then, conditions (\ref{eqn:7})-(\ref{eqn:9})  hold  for $\tau + \delta$ for any $\delta \ge 0$ with the same $m$.
\end{theorem}

\begin{proof}
Suppose that (\ref{eqn:7})-(\ref{eqn:9}) hold, and let  $V_i(x) = \max \{ x^T P_{i,r} \, x : r =1 , \cdots, m \}$. Consider any $ \delta \ge 0$. From  (\ref{eqn:8})  it follows that $\dot V_i(x ; A_i x) < 0$ and hence $V_i(x(\delta)) < V_i(x(0))$ for every $x(0)$, which implies that
$e^{A_i^T \delta} P_{i,r} e^{A_i \delta} \prec P_{i,r}$.
Now, pre- and post-multiply  (\ref{eqn:9}) by $e^{A_i^T \delta}$ and $e^{A_i \delta}$ respectively. It follows that

{\footnotesize
\begin{align*}
 e^{A_i^T (\tau +\delta)} P_{j,q} e^{A_i (\tau +\delta)} & \prec e^{A_i^T \delta} P_{i,r} e^{A_i \delta}  \\
 & +  \sum_{\substack{s=1\\s\ne r}}^{m} \gamma_{jqirs} e^{A_i^T \delta} (P_{i,s} - P_{i,r} ) e^{A_i \delta} \\
& \prec P_{i,r} + \sum_{\substack{s=1\\s\ne r}}^{m} \gamma_{jqirs} e^{A_i^T \delta} (P_{i,s} - P_{i,r} ) e^{A_i \delta}
\end{align*}
}
Thus, $V_j (e^{A_i^T (\tau +\delta)} x ) < V_i(x)$  and  the theorem holds.
\end{proof}

Another important property of conditions of Theorem \ref{thm:2} that allows us to calculate $\tau_{[m]} $ is stated in the following lemma.

\begin{lemma}
\label{lem:1}
For a given $\tau$, suppose   that   conditions (\ref{eqn:7})-(\ref{eqn:9}) are feasible for some $m$. Then,  they are also feasible for $m+1$.
\end{lemma}

\begin{proof}
Suppose a set of matrices $P_{i,r} \succ 0$, $r=1,\cdots, m$ satisfies conditions of Theorem \ref{thm:2} for a given $\tau$.   A  feasible solution for the case of $m+1$ is obtained by setting $P_{i,m+1} = P_{i,m}$, and keeping the rest of the $P_{i,r}$'s the same.
\end{proof}

An immediate conclusion of Lemma \ref{lem:1} and Theorem \ref{thm:3}, is that the sequence of $\tau_{[m]}$, $m = 1,2, \cdots $ is a  non-increasing sequence  and the limit
$\bar \tau_{[m]} := \lim_{m \rightarrow \infty} \tau_{[m]}$ exists. Obviously, $\bar \tau_{[m]}$ is an upper bound on the minimum dwell time, i.e.
$\bar \tau_{[m]} \ge \tau_{min}$.

\subsection{Numerical Solution of Conditions of Theorem \ref{thm:2}}
The stability conditions (\ref{eqn:7})-(\ref{eqn:9}) appeared in Theorem \ref{thm:2} are bilinear matrix inequalities with respect to variables
$P_{i,r}$, $\alpha_{irs}, \gamma_{jirs}$. Of course linearity is recovered if one fixes the matrices $P_{ir}$ or the scalars $\a$ and $\g$.
Unfortunately, nested iterations by iteratively fixing $P_{ir}$ and computing proper values of $\a, \g$'s by means of convex optimization does not converge.
In fact, it is known that finding the global solutions of BMI problems are NP-hard.
We, however,  find out that using the path-following method \cite{path-following}, we can find feasible solutions for the BMI conditions (\ref{eqn:7})-(\ref{eqn:9}).  The basic idea of path-following method  is to use the first order approximation of the variables. To this end,  we perturb the variables  to $P_{ir} + \D P_{ir}$, $\a_{irs} + \D \a_{irs}$ and $\g_{jirs} + \D \g_{jqirs}$.  Then, by ignoring the higher order terms, the conditions of Theorem \ref{thm:2} become:
{\small
\begin{subequations}
\label{eqn:Path}
\begin{align}
&  P_{i,r} + \D P_{i,r} \succ 0 \\
&  A_{i}^{T} P_{i,r} + P_{i,r}  A_i + A_{i}^{T} \Delta P_{i,r} + \Delta P_{i,r} A_{i}  \prec \sum_{s \neq r}   \Big[  \alpha_{irs} (P_{i,s} - P_{i,r} )  \nonumber \\
& \qquad \quad + \alpha_{irs}  ( \Delta P_{i,s} - \Delta P_{i,r} ) + ( P_{i,s} -  P_{i,r} )  \Delta \alpha_{irs} \Big]  \\
&  e^{A_{i}^{T} \tau }  P_{j,r}  \, e^{A_{i} \tau }  +   e^{A_{i}^T \tau }  \Delta P_{j,r} \, e^{A_{i} \tau }  \nonumber \\
& \qquad \quad  \prec  \sum_{s\neq r} \Big[ \g_{jqirs} (P_{i,s} + \Delta P_{i,s}) +   \D \g_{i,r,s} P_{i,s} \Big]
\end{align}
\end{subequations}
}

For  given $P_{ir}, \a_{irs} , \g_{jqigs}$ the above conditions  are LMI with respect to variables $\D P_{ir}, \D\a_{irs} , \D\g_{jqirs}$.

One can start with a feasible solution to (\ref{eqn:7})-(\ref{eqn:9}) and then  solve  (\ref{eqn:Path}a)-(\ref{eqn:Path}c) for $\D P_{ir}, \D\a_{irs} , \D\g_{jqirs}$. The next step is to update the scalars by letting $\a_{irs} \leftarrow \a_{irs} +  \D\a_{irs}$ and $\g_{jirs} \leftarrow \g_{jqirs} +  \D\g_{jqirs}$. Now by fixing these variables, (\ref{eqn:7})-(\ref{eqn:9}) is an LMI in $P_{ir}$ and its feasibility can be checked efficiently. This iteration is then combined with a bisection search on $\tau$ to find the minimum $\tau_{[m]}$ for a given $m$. When  $\tau_{[m]}$   cannot be improved with a fixed $m$, we increase the $m$ until $\tau_{[m+1]} = \tau_{[m]}$.
Note that condition of Theorem \ref{thm:2} for $m=1$ is an LMI and the solution to that can be used for the initialization of the above iterative procedure.

%

\section{Example}
\label{sec:example}

To illustrate the effectiveness  of the proposed   method, an examples is presented in this section.
%
The  example is taken from \cite{Blanchini-2010} where the system matrices are
{\small
$$A_1=
\left[
\begin{array}{cc}
0  & 1     \\
 -10 &  -1    \\
\end{array}
\right]
, \;\;A_2=
\left[
\begin{array}{cc}
0  & 1     \\
 -0.1 &  -0.5    \\
\end{array}
\right]$$
}

It has already been seen that for this system the minimum dwell time is 2.7078 \cite{Blanchini-2010}.
The dwell-time obtained from method of \cite{Blanchini-2010} with a polyhedral characterization is $2.7420$.
The sequence of the upper bounds of $\tau_{min}$ obtained from our proposed method is  shown in Table \ref{table1}, where the minimum dwell time $\tau_{min}=2.70801$ can be obtained with $m=4$ quadratic functions.
The $P_{ir}$ are reported here for verification:\\
{\scriptsize
$P_{11} =
\left[
\begin{array}{cc}
196.665  & 17.322     \\
 * &  25.408   \\
\end{array}
\right]$,
$P_{12} =
\left[
\begin{array}{cc}
196.502  & 38.442     \\
 * &  12.137    \\
\end{array}
\right]$, \\
$P_{13} =
\left[
\begin{array}{cc}
187.555  &  4.997     \\
 * &  19.280    \\
\end{array}
\right]$,
$P_{14} =
\left[
\begin{array}{cc}
196.502  & 38.455    \\
 * &  11.973    \\
\end{array}
\right]$, \\
$P_{21}=P_{22}=P_{23}=P_{24}=
\left[
\begin{array}{cc}
137.512   & 177.991  \\
 * &  360.358    \\
\end{array}
\right]$
}

\begin{table}[h]
\begin{center}
\setlength{\tabcolsep}{5 pt}
\renewcommand{\arraystretch}{1}
\begin{tabular}{c|c c c c c}
$m$  &   1 & 2 & 3 & 4 & 5    \\ \hline
$\tau_{[m]}$ &  2.75090 & 2.70794 & 2.70782  & 2.70781 & 2.70781
\end{tabular}
\end{center}
\caption{The sequence of upper bounds of $\tau_{min}$}
\label{table1}
\end{table}

\begin{figure}[h]
\centering
\includegraphics[trim = 1mm 5mm 1mm 7mm, clip=true,width=.3\textwidth] {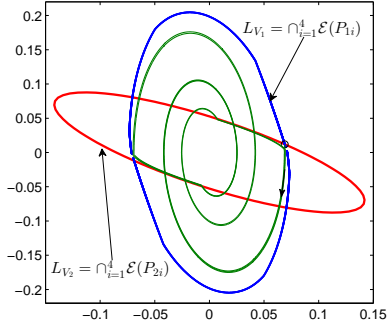}
\caption{Level set of $V_1, V_2$ for $m=4$.}
\label{fig:1}
\end{figure}

Figure \ref{fig:1} shows the level-sets of the Lyapunov functions, i.e. $L_{V_i} = \cap_{r=1}^4 \E (P_{ir})$ and a state trajectory from $x_0= (0.0689,0.0119)$ under periodic switching where $t_{k+1} - t_k = \tau_{[4]}$ and $\s(0) = 1$. The Lyapunov function $V_{\s}(x(t))$ for this trajectory is also shown in Fig. \ref{fig:2}(a). While $V_{\s}(x(t))$  increases at switching instants, the  sequence of $V_{\s}(x(t_k))$ is monotonically decreasing (see Fig. \ref{fig:2}(b)).

\begin{figure}[h]
\centering
\includegraphics[trim = 1mm 3mm 1mm 10mm, clip=true, width=.35\textwidth] {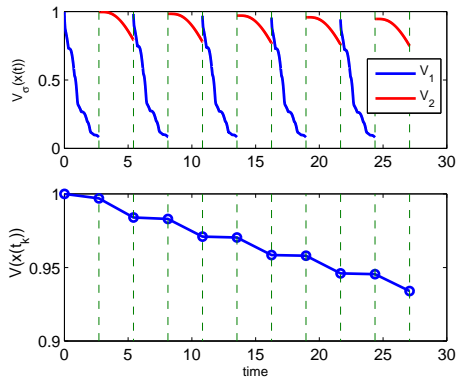}
\caption{Illustration of $V_{\s}(x(t)), V_{\s}(x(t_k))$.}
\label{fig:2}
\end{figure}

\section{Conclusions}
\label{sec:conclusions}
Piece-wise quadratic  functions are utilized to derive sufficient conditions for stability of switched linear systems under dwell-time switching.  The stability conditions are  in the form of  bilinear  matrix inequalities and are solved with path-following method.  By increasing the number of  quadratic functions, a sequence of upper bounds of the minimum dwell time is obtained.
Numerical examples suggested that the conditions has the potential to be also necessary provided that  the number of  quadratic functions is sufficiently large. Further investigation is required to prove the necessity of the proposed conditions.


\section*{Acknowledgment}

The financial support of A*STAR industrial robotic grant (Grant No. R-261-506-005-305) is gratefully acknowledged.



%

\bibliographystyle{IEEEtran}
\bibliography{ref_switch}

\end{document}